\documentclass{amsart}

\usepackage{amscd,amsmath,amssymb,amsthm,mathrsfs,latexsym}
\usepackage[pdftex]{graphicx}
\usepackage[mathcal]{euscript}
\usepackage[all]{xy}
\usepackage{color}
\usepackage{fancyhdr}
\usepackage{color}
\usepackage{graphicx}
\usepackage{caption}
\usepackage{subcaption}
\usepackage{color,soul} 
\usepackage{hyperref} 
\hypersetup{
 colorlinks=true, 	 
 linktoc=all,     	 
 linkcolor=black,  	 
 citecolor=black,    
 filecolor=magenta,  
 urlcolor=cyan       
}
\usepackage{tikz} 		
\usepackage{tikz-cd}
\usetikzlibrary{matrix,arrows,decorations.pathmorphing,intersections}

\usepackage{pifont} 
\usepackage{fourier-orns} 

\newtheorem{thm}{Theorem}[section]

\newtheorem{lemma}[thm]{Lemma}
\newtheorem{prop}[thm]{Proposition}
\newtheorem{cor}[thm]{Corollary}

\theoremstyle{definition}
\newtheorem{defn}[thm]{Definition}

\newtheorem{ex}[thm]{Example}
\newtheorem*{rmk}{Remark}

\newcommand{\Z}{{\mathbb{Z}}}

\newcommand{\Out}{{\rm Out}}
\newcommand{\Inn}{{\rm Inn}}
\newcommand{\Aut}{{\rm Aut}}
\newcommand{\IA}{{\rm IA}}

\newcommand{\colim}{{{\rm colim}\hspace{.2em}}}

\newcommand{\gr}{{{\rm gr}}}

\newcommand{\gw}{{{\gamma_{\omega}}}}


\def\ds{\displaystyle}
\numberwithin{equation}{section} 

\title{On monodromy representations in Denham-Suciu fibrations}

\author{Mentor Stafa}
\address{Department of Mathematics, Tulane University, New Orleans, LA 70118}
\email{mstafa@tulane.edu}



\thanks{The author was supported by DARPA grant number N66001-11-1-4132}

\begin{document}

\begin{abstract}
We study the monodromy representation corresponding to a fibration introduced by G. Denham and A. Suciu \cite{denham}, which involves polyhedral products given in Definition \ref{defn: polyhedral product}. 
Algebraic and geometric descriptions for these monodromy representations are given.
In particular, we study the case of a product of two finite cyclic groups and obtain representations into $\Out(F_n)$ and $SL_n(\Z)$. We give algebraic descriptions of monodromy for the case of a product of any two finite groups . Finally we give a geometric description for monodromy representations of a product of 2 or more finite groups to $\Out(F_n)$, as well as some algebraic properties. {The geometric description does not rely on choosing a basis for the fundamental group of the fibre in terms of commutators, hence avoids this delicate question.}
\end{abstract}

\maketitle

\tableofcontents

\section{Introduction}

\

Let $(X,A)$ be a pair of spaces and let $K$ be a simplicial complex with $n$ vertices. A new topological space can be constructed using the pair $(X,A)$ and $K$, called a \textit{polyhedral product} and denoted by $Z_K(X,A)\subset X^n$ (see Definition \ref{defn: polyhedral product}). Polyhedral products are actively studied and stand at the foundations of the field of \textit{toric topology}, see for example work of {A. Bahri}, M. Bendersky, F. Cohen and S. Gitler \cite{cohen.macs} for an introduction, or work of V. Buchstaber and T. Panov \cite{buch.panov.toric.topology} for a survey of toric topology. A short introduction on polyhedral products is given in Section \ref{sec: polyhedral products}.

\

Given a locally trivial fibration $f:E\to B$ with fibre $F$, there is an action of the fundamental group of $B$ on the fundamental group of the fibre $F$ and consequently the first homology of $F$. This action gives rise to a  representation called the monodromy representation. One natural use of this representation is calculating the homology of the total space $E$ {using the Serre spectral sequence, when the fundamental group of the base space $B$ acts non-trivially on the homology of F. In that case the homology of $F$ is a non-trivial module over the group ring $\Z \pi_1(B)$.}

\

{ Let $G$ be a topological group. Let $BG$ denote the classifying space of $G$ and $EG$ denote a contractible space on which $G$ acts freely and properly discontinuously. The projection $EG \to BG=EG/G$ is then a bundle projection. In particular, if $G$ is a finite discrete group, then $EG$ is the universal cover of $BG$. 
G. Denham and A. Suciu \cite[Lemma 2.3.2]{denham} gave a natural fibration relating the polyhedral product for the pair $(BG,\ast)$, where $\ast$ is the basepoint of $BG$, and the polyhedral product for $(EG,G)$. That is, for a simplicial complex $K$ with $n$ vertices  the polyhedral product $Z_K(BG,\ast)$ fibres over the product $(BG)^n$ as follows
\begin{equation}\label{D-S-fibration INTRO}
Z_K(EG,G) \to (EG)^n \times_{G^n} Z_K(EG,G) \to (BG)^n ,
\end{equation}
where $(EG)^n \times_{G^n} Z_K(EG,G)\simeq Z_K(BG,\ast)$. The group $G$ acts on the space $Z_K(EG,G)\subset (EG)^n$ coordinate--wise, thus there is an action of $G^n$ on the fibre $Z_K(EG,G)$. This fibration generalizes previous constructions in work of M. Davis and T. Januszkiewicz \cite{davis.januszckiewicz} and work of V. Buchstaber and T. Panov \cite{buch.panov}. In particular, it originates from the Davis-Januszkiewicz space defined by
$$\mathcal{DJ}(K)= E(S^1)^n \times_{(S^1)^n}Z_K(D^2,S^1),$$
and generalizes the result of V. Buchstaber and T. Panov that the homotopy fibre of the inclusion 
$$Z_K(BS^1,\ast) \hookrightarrow (BS^1)^n$$ 
is homotopy equivalent to $Z_K(ES^1,S^1)$.

\
}

This paper studies the monodromy representation of {the natural fibration in equation \ref{D-S-fibration INTRO} introduced by G. Denham and A. Suciu \cite{denham}}, where all the spaces are polyhedral products of special pairs of spaces. The monodromy action is then described naturally using the geometry of the fibre, which arises from properties of polyhedral products. We will use convenient models for the homotopy type of the pairs $(X,A)$ to achieve this. In certain cases we will be able to give precise algebraic descriptions of the action.

\

The fibration in equation \ref{D-S-fibration INTRO} can be slightly generalized if we allow for the group $G$ to vary in each coordinate. In a similar way, one can define a polyhedral product for a sequence of pairs $(\underline{X},\underline{A})=\{(X_i,A_i)\}_{i=1}^n$ and a simplicial complex $K$, and denote it by $Z_K(\underline{X},\underline{A})$, see Definition \ref{defn: polyhedral product}. Two such sequences of pairs are $(\underline{BG},\underline{\ast})=\{(BG_i,\ast_i)\}_{i=1}^n$ and $(\underline{EG},\underline{G})=\{(EG_i,G_i)\}_{i=1}^n$. Then there is a bundle 
$$
Z_K(\underline{EG},\underline{G}) \to (\prod_{i=1}^n EG_i) \times_{(\prod_{i=1}^n G_i)} Z_K(\underline{EG},\underline{G}) \to \prod_{i=1}^n BG_i ,
$$
where $(\prod_{i=1}^n EG_i) \times_{(\prod_{i=1}^n G_i)} Z_K(\underline{EG},\underline{G}) \simeq Z_K(\underline{BG},\underline{\ast})$. Therefore, we can rewrite the fibration in (\ref{D-S-fibration INTRO}) as follows
\begin{equation}\label{D-S-fibration INTRO 2}
Z_K(\underline{EG},\underline{G}) \to Z_K(\underline{BG},\underline{\ast}) \to \prod_{i=1}^n BG_i.
\end{equation}
Similarly, $G_1\times \cdots\times G_n $ acts on the fibre $Z_K(\underline{EG},\underline{G})\subset EG_1\times \dots \times EG_n$ coordinate-wise. We will refer to this fibration as the \textit{Denham--Suciu fibration}. 

\

Now suppose that $G_1,\dots , G_n$ are discrete groups. One can study the action of the fundamental group of the base space on the fundamental group of the fibre, namely the action of $G_1\times \cdots\times G_n $ on $\pi_1(Z_K(\underline{EG},\underline{G}))$. A natural question is also determining the structure of the first homology group of $Z_K(\underline{EG},\underline{G})$ as a module over the group ring $\Z [ G_1\times \cdots\times G_n $].

\

There are cases where the fundamental group of the fibre can be given explicitly. In particular, if $K=K_0$ is the zero simplicial complex consisting of only $n$ vertices and no edges, and if $G_1,\dots , G_n$ are finite discrete groups, then it is shown in \cite[Theorem 3.8]{stafa.fund.gp} that $\pi_1(Z_{K_0}(\underline{EG},\underline{G}))\cong F_{N_n}$, the free group of rank $N_n$. The natural number $N_n$ is shown in \cite[Corollary 3.7]{stafa.fund.gp} to be 
\begin{equation}\label{eqn: formula for N_n}
N_n=(n-1)\prod_{i=1}^n m_i - \sum_{i=1}^n (\prod_{j\neq i} m_j)+1,
\end{equation}
where $m_i=|G_i|$. Note that the rank of the free group in this case depends only on the order of the groups $G_i$. Moreover, for the special case of $K_0$ all the spaces in the Denham--Suciu fibration are Eilenberg--Mac Lane spaces, see \cite[Theorem 1.1]{stafa.fund.gp}, therefore there is a short exact sequence of groups
$$
1 \to F_{N_n} \to \pi_1(Z_{K_0}(\underline{BG},\underline{\ast})) \to \prod_{i=1}^n G_i \to 1.
$$
By definition of polyhedral products we have $Z_{K_0}(\underline{BG},\underline{\ast})=BG_1 \vee \cdots \vee BG_n$, see Example \ref{example: polyhedral products}. Hence, we get
$$
1 \to F_{N_n} \to G_1 \ast \cdots \ast G_n \to \prod_{i=1}^n G_i \to 1,
$$
where $G\ast H$ denotes the free product of the groups $G$ and $H$. Note that $F_{N_n}$ is generated by commutators in the free product $G_1 \ast \cdots \ast G_n$. In fact this is what makes monodromy action a delicate question. There is no obvious ``nice'' basis for the free group $F_{N_n}$ in terms of commutators that makes the algebraic computations of the monodromy representation accessible. To avoid this problem, we replace the pairs $(EG_i,G_i)$ with $([0,1],F_i)$, where $F_i$ is a subset of the unit interval $[0,1]$ with the same cardinality as $G_i$, and give the monodromy action geometrically for the general case, see Section \ref{section.monodromy}. This is possible since up to homotopy polyhedral products depend only on the relative homotopy type of the pairs $(X,A)$.

\

For $F_n$ the free group of rank $n$ let $\Aut(F_n)$ and $\Inn(F_n)$ denote the group of automorphisms and the group of inner automorphisms of $F_n$, respectively. Let $\Out(F_n):= \Aut(F_n)/\Inn(F_n)$ denote the group of outer automorphisms of $F_n$. Then the monodromy action for the finite discrete groups $G_1,\dots,G_n$ and the zero simplicial complex $K_0$ is given by the representation
$$
\rho_{K_0}: G_1\times \cdots\times G_n \to \Out(F_{N_n}),
$$ 
see Section \ref{section.monodromy.action}.

\

Similarly, in \cite[Theorem 1.1]{stafa.fund.gp} it is shown that $Z_K(\underline{BG},\underline{\ast})$ is an Eilenberg--Mac Lane space if and only if $K$ is a \textit{flag complex} (see Definition \ref{defn: flag complex}), and if $K$ is a flag complex there is also a representation
$$
\rho_K : G_1 \times \cdots \times G_n  \to \Out(\pi_1(Z_K(\underline{EG},\underline{G}))),$$
see Section \ref{section.monodromy.action}. Note that the zero simplicial complex $K_0$ is a special case of a flag complex. The computations will be restricted to $K_0$, and in the last section we will discuss the representations for other choices of $K$.

\

We obtain the following results regarding monodromy representations.

\begin{prop}\label{proposition: faithful rep of GxH to Out(Fn) INTRO}
Let $G=\{g_1,\dots ,g_m\}$ and $H=\{h_1,\dots,h_n\}$ be two finite discrete groups of order $m$ and $n$, respectively. Then the monodromy action is given by the faithful representation 
$$
\varphi : G\times H \to \Out(F_{k})
$$
where for any $t \in G\times H$, the image $\varphi(t)=\varphi_t$ is given by the following equations 
\begin{equation}\label{eqn: monodromy for G and H INTRO}
\begin{split}
&\varphi_{g_k}([g_i,h_j])={g_k}[g_i,h_j]{g_k}^{-1}=[g_k g_i,h_j][h_j,g_k]  \\ 
&\varphi_{h_k}([g_i,h_j])={h_k}[g_i,h_j]{h_k}^{-1}=[h_k,g_i][g_i,h_k h_j],
\end{split}
\end{equation}
and $k=(m-1)(n-1)$.
\end{prop}

As a special case we work out the case of two finite cyclic groups explicitly.

\begin{thm}\label{thm: 2 cyclic gps rep}
Let $G_1=C_n$ and $G_2=C_m$ be two finite cyclic groups of order $n$ and $m$, respectively. Then the monodromy action is given by a representation 
$$
C_n \times C_m \to \Out(F_{k}),
$$
where $k=(n-1)(m-1)$, which gives a faithful representation
$$ 
 C_n \times C_m \to SL_k(\Z),
$$
given by equations \ref{monodromy equation 1} and \ref{monodromy equation 2}.
\end{thm}\

The method presented in this paper applies to any finite collection of finite discrete groups $G_1,\dots,G_n$. Let $\gr_{\ast}(G)$ denote the Lie algebra associated to the descending central series of the group $G$. Finally, we give some properties of the representation
$$
G_1 \times \cdots \times G_k \to \Out(F_{N_k})
$$
for finite discrete groups.

\begin{lemma}\label{lemma: properties of the reps}
Let $\{G_i\}_{i=1}^n$ be a collection of finite discrete groups and $K_0$ be the $0$--simplicial complex on $n$ vertices. Let $\rho :\prod_{i=1}^n G_i \to  \Out(F_N)$ be the monodromy representation where $F_N$ is isomorphic to the kernel of the projection 
$$p: G_1 \ast \cdots \ast G_n \to \prod_{i=1}^n G_i.$$ 
Then the following properties of $\rho$ hold:
\begin{enumerate}
\item There is a choice of a generating set for $F_N$ that consists of elements of the form
\[f=[g_{i_1},[g_{i_2},[\dots,[g_{i_{k-1}},g_{i_{k}}]\dots]]] \in \Gamma^k(G_1 \ast \cdots \ast G_n)\]
such that $g_{i_j} \in G_{i_j}$, for all $i_j$.
\item For any $g \in G_1 \ast \cdots \ast G_n$, the map $\rho(g) \in \Aut(F_N)$ satisfies $\rho(g)(f)=\Delta \cdot f$, {where} $\Delta \in \Gamma^{k+1}(G_1 \ast \cdots \ast G_n)$. That is, $\Delta$ is trivial in $ \gr_{p}(G_1 \ast \cdots \ast G_n)$ for $p \leq k$.
\end{enumerate}
\end{lemma}

\

In the last part of the paper we investigate the possible implications of the representation
$$
\rho_{K_0} :\prod_{i=1}^n G_i \to  \Out(F_{N_n}),
$$
corresponding to the zero simplicial complex with $n$ vertices, on the representations
$$
\rho_K :\prod_{i=1}^n G_i \to  \Out(\pi_1(Z_K(\underline{EG},\underline{G}))),
$$
where $K$ is an arbitrary simplicial complex on $n$ vertices. The motivation comes from the fact that there is a homotopy equivalence $Z_K(\underline{EG},\underline{G})\simeq Z_K({[0,1]},\underline{F})$. We are able to reduce the question to the existence of certain commutative diagrams. More precisely the question is reduced to the existence of a map $r:\Out(F_{N_n}) \to \Out(\pi)$, where $\pi=\pi_1(Z_K(\underline{EG},\underline{G}))$, such that the following diagram commutes
\begin{center}
\begin{tikzcd}
{\color{white}{1}} & 
\Out(F_{N_n})  \arrow[dotted]{dd}{r} \\
G_1 \times \cdots \times G_n  \arrow{ur}{\rho_{K_0}}\arrow{dr}[swap]{\rho_{K}} & \\
& \Out(\pi) .
\end{tikzcd}
\end{center}

\

\section{Polyhedral products}\label{sec: polyhedral products}

\

In this section we give a short introduction to polyhedral products. Let $(\underline{X},\underline{A} )$ denote a finite sequence of pointed $CW$--pairs $\{(X_i,A_i)\}_{i=1}^n$. Let $[n]$ denote the set of natural numbers $\{1,2,\dots,n\}$.

\begin{defn}
A \textit{simplicial complex} $K$ on the set $[n]$ is a subset of the power set of ${[n]}$, such that, if $\sigma \in K$ and $\tau \subseteq \sigma$ then $\tau \in K$.
\end{defn}

A simplex $\sigma \in K$ is given by an increasing sequence of integers 
$$\sigma=\{i_1, i_2, \dots ,i_q\},$$
where $1 \leq i_1 < i_2 < \cdots < i_q \leq n$. In particular, the empty set $\varnothing$ is an element of $K$. The \textit{geometric realization} $|K|$ of $K$ is a simplicial complex inside the simplex $\Delta[n-1]$. 

\

Define a functor
$$D: K \longrightarrow CW_{\ast},$$ 
where $CW_{\ast}$ denotes the category of pointed $CW$--complexes. For any $\sigma \in K$ let 

$$\ds D(\sigma)=\prod_{i=1}^n Y_i=Y_1 \times \cdots \times Y_n \subseteq \prod_{i=1}^n X_i,$$
where
$$
  Y_i = \left\{\def\arraystretch{1.2}%
  \begin{array}{@{}c@{\quad}l@{}}
    A_i & : i \notin \sigma ,\\
       X_i & : i \in \sigma .\\
  \end{array}\right.
$$

\begin{defn}\label{defn: polyhedral product}
The \textit{polyhedral product}  $Z_K(\underline{X},\underline{A})$ is the subset of the product $X_1 \times \cdots \times X_n$ given by the colimit
\[ Z_K(\underline{X},\underline{A})= \underset{{\sigma \in K}}{\colim}\hspace{.05in} D(\sigma) = \bigcup_{\sigma \in K} D(\sigma)  \subseteq \prod_{i=1}^n X_i, \]
where the maps are the inclusions and the topology is the subspace topology.
\end{defn}

In other words the polyhedral product is the colimit of the diagram of spaces $D(\sigma)$. The following notations appear throughout the literature and all represent the same polyhedral product: $Z_K(\underline{X},\underline{A})$, $Z_K(X_i,A_i)$, $Z(K;(\underline{X},\underline{A}))$ and $(\underline{X},\underline{A})^K$. 
If the sequence of pairs is constant then we simply write $Z_K({X},{A})$.

\begin{ex}\label{example: polyhedral products}
Assume $K$ is the zero simplicial complex $\{ \{1\},\dots,\{n\} \}$, $X_i=X$ and $A_i$ be the basepoint $\ast \in X$. Then 
$$Z_K(\underline{X},\underline{A})= Z_K(X,\ast)=X \vee \dots \vee X,$$
the $n$--fold wedge sum of the space $X$. On the other hand, if $K$ is the full simplex, then 
$$ Z_K(X,\ast)= X_1 \times \cdots \times X_n.$$
\end{ex}

\begin{ex}
Assume $K=\{ \{1\},\{2\} \}$. Let $(\underline{X},\underline{A})=(D^n,S^{n-1})$, the pair consisting of an $n$--disk and the bounding $(n-1)$--sphere. Then 
\[Z_K(\underline{X},\underline{A})= Z_K(D^n,S^{n-1}) =D^n \times S^{n-1} \cup S^{n-1}\times D^n=\partial D^{2n}=S^{2n-1}.\]
\end{ex}

\begin{defn}\label{defn: graph product}
Given a simplicial graph $\Gamma$ with vertex set $S$ and a family of groups $\{G_s\}_{s\in S}$, their \textit{graph product} 
\[\ds \prod_{\Gamma} G_s \]
is the quotient of the free product of the groups $G_s$ by the relations that elements of $G_s$ and $G_t$ commute whenever $\{s, t\}$ is an edge of $\Gamma$. Note that if $\Gamma$ is the complete graph on $n$ vertices, then 
$\prod_{\Gamma} G_s \cong \prod_{i=1}^n G_i.$
\end{defn}

\begin{defn}\label{defn: flag complex}
$K$ is a \textit{flag complex} if any finite set of vertices, which are pairwise connected by edges, spans a simplex in $K$.
\end{defn}

\section{Monodromy representation}\label{section.monodromy}

Let $G_1,\dots,G_n$ be finite discrete groups of order $|G_i|=m_i$, for $1\leq i \leq n$. In this section we are interested in describing the monodromy representation corresponding to the Denham--Suciu fibration (equation \ref{D-S-fibration INTRO 2}) given by
\[
Z_K(\underline{EG},\underline{G}) \longrightarrow Z_K(\underline{BG}) \longrightarrow \prod_{i=1}^n {BG_i}.
\]
Recall that the homotopy type of the polyhedral product $Z_K(\underline{X},\underline{A})$ depends only on the relative homotopy type of the pairs $(\underline{X},\underline{A})$. 
\begin{lemma}\label{lemma: rel homotopy equiv}
Let $G$ be a finite discrete group of order $m$. Then there is a relative homotopy equivalence $(EG,G)\sim ([0,1],F)$, where $F$ is a subset of $[0,1]$ of cardinality $m$.
\end{lemma}
\begin{proof}
See \cite[Lemma 3.5]{stafa.fund.gp}.
\end{proof}

Hence, there is a homotopy equivalence $Z_K(\underline{EG},\underline{G})\simeq Z_K(\underline{I},\underline{F})$, where $I$ is the unit interval $[0,1]$ and $(\underline{I},\underline{F})= \{(I,F_j)\}_{j=1}^n$. If $K=K_0$ is the zero skeleton of the $(n-1)$--simplex, then $Z_{K_0}(\underline{I},\underline{F})$ is a connected simplicial graph embedded in the space $[0,1]^n \subset \mathbb{R}^n$ (see Figure \ref{fig:2dcase}) and hence, has the homotopy type of a wedge of $N_n$ circles. As mentioned in the introduction and in \cite[Proposition 3.6]{stafa.fund.gp}, the integer $N_n$ is given by 
\[
N_n=(n-1)\prod_{i=1}^n - \sum_{i=1}^n \big( \prod_{j\neq i} m_j \big) +1.
\]

Lemma \ref{lemma: rel homotopy equiv} allows for a geometric description of the fibre in the Denham--Suciu fibration for any $K$, and for the case of the zero simplicial complex in particular. This is a geometric model that will sometimes be used to describe monodromy concretely. This description of the fibre depends only on the order of the groups $G_i$, but clearly the monodromy representations still depend on the structure of the groups. For some computations we will restrict to finite cyclic groups, but we will also describe the method for other finite discrete groups. 

\

\subsection{{Generators for the fundamental group}}\label{section.monodromy.generators}

\

\

Let $K_0$ denote the 0--simplicial complex on $n$ vertices. In this section we describe explicit loops in $Z_{K_0}(\underline{I},\underline{F})$, whose equivalence classes constitute a generating set for the fundamental group, $F_{N_n}$. Recall that the simplicial complex $K_0$ is a flag complex, thus the spaces in the Denham--Suciu fibration are Eilenberg--Mac Lane spaces and there is a short exact sequence of groups 
\[ 1 \longrightarrow F_{N_n} \longrightarrow \ G_1 \ast \cdots \ast G_n \longrightarrow G_1 \times \cdots \times G_n \longrightarrow 1, \]
where $\ G_1 \ast \cdots \ast G_n $ denotes the free product of the groups. The classes of loops that we will find are therefore generators for the free group $F_{N_n}$ in the short exact sequence.

\

The homotopy type of $Z_K(\underline{EG},\underline{G})$ depends only on the cardinality of $G_i$. Hence, when finding the loops for $Z_K(\underline{EG},\underline{G})$ for finite cyclic groups $G_1,\dots,G_n$, the same computation holds for any collection of groups with the same order, that is the same classes of loops will be used to describe the monodromy. However, the representation depends on the structure of the groups.

\begin{defn}\label{defn: loops described by words w}
The loops are described as follows: {Let $\ast=(0,\dots,0)\in I^n$ be the basepoint of $Z_{K_0}(\underline{I},\underline{F})$, which is the image of $(1_{G_1},\dots,1_{G_n})$ under the homotopy equivalence}. Starting from the basepoint $\ast$, each path in $Z_{K_0}(\underline{I},\underline{F})$ will be tracked by a word $\omega=x_{i_1}^{j_1}x_{i_2}^{j_2} \cdots x_{i_r}^{j_r}$, where $x_{i_k}^{j_k} \in G_k$, each letter $x_k$ showing the coordinate of the group it belongs to, together with the exponent $j_k$ showing the distance taken in that direction. See Figure \ref{fig:2dcase} for a picture in two dimensions. For any word $\omega \in F_{N_n}$, let $\gamma_{\omega}$ denote the path in $Z_{K_0}(\underline{I},\underline{F})$ tracked by the word $\omega$.
\end{defn}

\begin{lemma}
The path tracked by the word $x_{i_1}^{j_1}x_{i_2}^{j_2} \cdots x_{i_r}^{j_r}$ is closed if and only if $\sum_{i=1}^r j_i=0 $.
\end{lemma}
\begin{proof}
This can be seen by arguing that, to start and end at the basepoint $\ast$, if $x_{i_1}^{j_1}$ is a letter of the word, then the letter $x_{i_1}^{-j_1}$ should also appear in the same word, otherwise one can never come back to $\ast$. Conversely if the sum $\sum_{i=1}^r j_i = 0$, then every move forward has been compensated by a move backward.
\end{proof}

\

\subsection{{Generators for the case of two finite groups}}

\

\

Let $G_1$ and $G_2$ be finite cyclic groups with order $m$ and $n$, respectively, such that $G_1=\{1,x_1,x_1^2,\dots,x_1^{m-1} \}$ and $G_2=\{1,x_2,x_2^2,\dots,x_2^{n-1} \}$. The zero simplicial complex with two vertices is $K_0=\{\{1\},\{2\}\}$. {Assume there are bijections of finite sets $G_1\approx F_1$ and $G_2\approx F_2$ given by
\begin{align*}
&G_1=\{1,x_1,x_1^2,\dots,x_1^{m-1} \} \approx F_1=\{0=t_{1,0}<t_{1,1}<\dots <t_{1,m-1} =1\} \subset I, \\
&G_2=\{1,x_2,x_2^2,\dots,x_2^{n-1}\} \approx F_2=\{0=t_{2,0}<t_{2,1}<\dots <t_{2,n-1} =1\} \subset I,
\end{align*}
identifying $x_i^k$ with $t_{i,k}$.}
Then from Definition \ref{defn: polyhedral product} we have the following
$$Z_{K_0}(\underline{EG},\underline{G})\simeq Z_{K_0}(\underline{I},\underline{F})= D(\{1\}) \cup D(\{2\}) = I\times F_2 \cup F_1 \times I, $$ 
see Figure \ref{fig:2dcase}. Consider the cycles $\gamma_{\omega}$ {described in Definition \ref{defn: loops described by words w} starting at the basepoint $\ast=(0,0)$}, given by the words 
$$\omega=[x_1^i,x_2^j]=x_1^ix_2^jx_1^{-i}x_2^{-j},$$ 
\begin{figure}[htbp]
\centering{
\begin{minipage}{0.5\linewidth}
\centering
{
\begin{tikzpicture}[scale=.5]
\draw [help lines,solid] (0,0) grid (9,8);
\draw [->,ultra thick](0,0) -- (0,6);
\draw [->,ultra thick](0,6) -- (5,6);
\draw [->,ultra thick](5,6) -- (5,0);
\draw [->,ultra thick](5,0) -- (0,0);
\draw (0,0) node {$\bullet$};
\coordinate [label=below:${(0,0)}$] (1) at (0,0);
\coordinate [label=below:$t_{1,5}$] (1) at (5,0);
\coordinate [label=left:$t_{2,6}$] (1) at (0,6);

\end{tikzpicture}
}
 \caption{2-dim. case, the loop $[x_1^6,x_2^5]$}
   \label{fig:2dcase}
\end{minipage}%
\begin{minipage}{0.5\linewidth}
\centering
{
\begin{tikzpicture}[scale=.5]
\draw (0,0) node {$\bullet$};
\draw [help lines,solid] (0,0) grid (9,8);
\draw [->,ultra thick](0.06,0) -- (0.06,2);
\draw [->,ultra thick](0.06,2) -- (5,2);
\draw [->,ultra thick](5,2) -- (5,0);
\draw [->,ultra thick](5,0) -- (0.06,0);

\coordinate [label=below:${(0,0)}$] (1) at (0,0);
\coordinate [label=left:$g$] (1) at (0,4);

\draw [->,ultra thick, dashed, gray](0,0) -- (0,6);
\draw [->,ultra thick, dashed, gray](0,6) -- (5,6);
\draw [->,ultra thick, dashed, gray](5,6) -- (5,4);
\draw [->,ultra thick, dashed, gray](5,4) -- (0,4);
\end{tikzpicture}
}
\caption{$g=x_2^4$ acting on $[x_1^2,x_2^5]$}
\label{fig:action}
\end{minipage}
}
\end{figure}

{\noindent where $1\leq i \leq m-1$ and $1 \leq j \leq n-1$. The following lemma tells which loops suffice.}
\begin{lemma}\label{lemma:elements.generating.set.n=2}
The set of words $\mathcal{W}=\{[x_1^i,x_2^j]|1\leq i \leq m-1, 1 \leq j \leq n-1\}$ generates all the cycles $\gw \in Z_{K_0}(\underline{I},\underline{F})$.
\end{lemma}
\begin{proof}

{Take an arbitrary word of finite length
$$x^{m_1}y^{n_1}x^{m_2}y^{n_2} x^{m_3}y^{n_3}\cdots x^{m_k}y^{n_k}x^{m_{k+1}}.$$
Then it can be written as a product of commutators as follows:
\begin{align*}
x^{m_1}y^{n_1}&x^{m_2}y^{n_2} x^{m_3}y^{n_3}\cdots x^{m_k}y^{n_k}x^{m_{k+1}}\\
	=&[x^{m_1},y^{n_1}]\cdot [y^{n_1},x^{m_1+m_2}]\cdot [x^{m_1+m_2},y^{n_1+n_2}]\cdot [y^{n_1+n_2},x^{m_1+m_2+m_3}]\\
	&\cdot [x^{m_1+m_2+m_3},y^{n_1+n_2+n_3}]\cdots [y^{n_1+\cdots n_k},x^{m_1+\cdots + m_k+m_{k+1}}].
\end{align*}
Since any cycle can be described by such a word, this suffices.
}
\end{proof}

\begin{lemma}\label{lemma:no.generating.set.n=2}
The set of words $\mathcal{W}=\{[x_1^i,x_2^j]|1\leq i \leq m-1, 1 \leq j \leq n-1\}$ is a minimal generating set.
\end{lemma}
\begin{proof}
First note that $|\mathcal{W}|=(m-1)(n-1)=mn - (m  + n)+1$. Then it follows from \cite[Proposition 3.6]{stafa.fund.gp} that $N_2=|\mathcal{W}|$.
\end{proof}

Now let $H_1$ and $H_2$ be finite discrete groups with cardinality $m$ and $n$, respectively. That is, $H_1=\{1,g_1,\dots,g_{m-1}\}$ and $H_2=\{1,h_1,\dots,h_{n-1}\}$. 
\begin{cor}
The set of words 
$$\mathcal{W}'=\{[g_i,h_j]|1\leq i \leq m -1, 1 \leq j \leq n-1\}$$ 
generates all the cycles in $Z_{K_0}(\underline{EH},\underline{H})\simeq Z_{K_0}(\underline{I},\underline{F})$. Moreover, this is a minimal generating set.
\end{cor}
\begin{proof}
{Take an arbitrary word of finite length
$$g_{m_1}h_{n_1}g_{m_2}h_{n_2} g_{m_3}h_{n_3}\cdots g_{m_k}h_{n_k}g_{m_{k+1}}.$$
Then it can be written as a product of commutators as follows:
\begin{align*}
g_{m_1}h_{n_1}&g_{m_2}h_{n_2} g_{m_3}h_{n_3}\cdots g_{m_k}h_{n_k}g_{m_{k+1}}\\
	=&[g_{m_1},h_{n_1}]\cdot [h_{n_1},g_{m_1}g_{m_2}]\cdot [g_{m_1}g_{m_2},h_{n_1}h_{n_2}]\cdot [h_{n_1}h_{n_2},g_{m_1}g_{m_2}g_{m_3}]\\
	&\cdot [g_{m_1}g_{m_2}g_{m_3},h_{n_1}h_{n_2}h_{n_3}]\cdots [h_{n_1}\cdots h_{n_k}, g_{m_1}\cdots g_{m_{k+1}}].
\end{align*}
Since any cycle can be described by such a word, this suffices. Now we have
$|\mathcal{W}'|=(m-1)(n-1)=mn - (m  + n)+1$. Then it follows from \cite[Proposition 3.6]{stafa.fund.gp} that $N_2=|\mathcal{W}'|$, thus giving minimality.
}

\end{proof}

The next step is to describe the action $G_1 \times \cdots \times G_n$ on these generators. We know $G_1 \times \cdots \times G_n$ {acts on the loops in the fiber} by conjugation, that is, 
$$g \cdot \gw = \gamma_{gwg^{-1}},$$
If we refer to {Figure \ref{fig:action}}, this action shifts the loop by $g$. For example, let $G_1=C_{10}$ and $G_2=C_9$ be the cyclic groups of order 10 and 9, respectively. The element $x_2^4 \in G_2 $ acts on the word $\omega =  [x_1^2,x_2^5]$ by conjugation
\[x_2^4 \cdot \omega = x_2^4 \cdot [x_1^2,x_2^5] = x_2^4 \omega x_2^{-4}.\]
Therefore
\[x_2^4 \cdot \gw = \gamma_{x_2^4 \omega x_2^{-4}},\]
which is the loop $\gw$ shifted by $x_2^4 \in G_2$ in the direction of $G_2$.

\

\subsection{{The monodromy action}}\label{section.monodromy.action}

\

Assume $G_1,\dots , G_n$ are finite discrete groups. Let $\mathcal{W}$ be a minimal set of generators for $\pi_1(Z_{K_0}(\underline{I},\underline{F}))\cong F_{N_n}$. Let $[\gw]=w\in F_{N_n}$ be the homotopy class of the loop $\gw$ in $Z_{K_0}(\underline{I},\underline{F})$, where $\omega \in \mathcal{W}$. Then $G_1 \times \cdots \times G_n$ {acts on the fundamental group of the fibre} $Z_{K_0}(\underline{I},\underline{F})$ as follows
$$
g\cdot \gw = \gamma_{g \omega g^{-1}},
$$
that is,
\[g\cdot \omega = g\cdot [\gw] = [\gamma_{g \omega g^{-1}}]={g \omega g^{-1}}. \]

\

The goal here is to write $g\omega g^{-1}$ as a product of words in $\mathcal{W}$. Then any element $g$ in the free product $G_1 \ast \cdots \ast G_n$ gives an automorphism of $F_{N_n}$, the free group with generators the elements of $\mathcal{W}$
\begin{align*}
 G_1 \ast \cdots \ast G_n & \overset{\varphi}{\longrightarrow} \Aut(F_{N_n})\\
   g & \longmapsto \varphi_g,
\end{align*}
where $\Aut(G)$ is the group of group automorphisms of $G$, under composition. One example is given in Section \ref{section:two cyclic groups}.

\

In general, recall that given a short exact sequence of discrete groups 
$$1\to A\to B\to C \to 1,$$ 
{with $A$ a normal subgroup of $B$}, there is a map
\begin{align*}
 B & \overset{\Theta}{\longrightarrow} \Aut(A)\\
   g & \longmapsto \Theta(g)
\end{align*}
such that $\Theta(g)(h)=ghg^{-1}$. There is also a map
\begin{align*}
 B & \overset{\Psi}{\longrightarrow} {\Inn(A)}\\
   g & \longmapsto \Psi(g)
\end{align*}
such that $\Psi(g)(h)=ghg^{-1}$, where $\Inn(A)$ is the group of \textit{inner automorphisms} of $A$. Moreover, $\Inn(A) \unlhd \Aut(A)$ and $\Out(A):=\Aut(A)/\Inn(A)$ is the group of \textit{outer automorphisms} of $A$. Note that for $A=F_n$ a free group, $F_n \cong \Inn(F_n)$. 

\

For the free group $F_n$ and $n \geq 2$, there  is a short exact sequence of groups
\[1 \longrightarrow \Inn(F_n) \longrightarrow \Aut(F_n) \longrightarrow \Out(F_n)\longrightarrow 1\]
and hence, a commutative diagram
\begin{center}
\begin{tikzcd}
 1 \arrow{r} & F_{N_n} \arrow{r} \arrow{d}{\Psi} & \ G_1 \ast \cdots \ast G_n \arrow{r} \arrow{d}{\Theta} & G_1 \times \cdots \times G_n \arrow{r} \arrow{d}{{\widetilde{\Theta}}} & 1 \\
 1 \arrow{r} & \Inn(F_{N_n}) \arrow{r} &\Aut(F_{N_n}) \arrow{r} & \Out(F_{N_n}) \arrow{r} & 1,
\end{tikzcd}
\end{center}
where $G_1,\dots , G_n$ are finite discrete groups. So the map 
$$\Theta : G_1 \ast \cdots \ast G_n \to \Aut(F_{N_n})$$ 
induces a map 
$${\widetilde{\Theta}: G_1 \times \cdots \times G_n \to \Out(F_{N_n}),}$$
which is the representation we are interested in.

\

There is also another short exact sequence 
\[1 \longrightarrow \IA_n \longrightarrow \Aut(F_n) \overset{\text{ab}}{\longrightarrow} GL_n(\mathbb{Z}) \longrightarrow 1,\]
with kernel the group $\IA_n$, which is the subgroup of automorphisms that restrict to the identity in the abelianization of $F_n$, and ``ab'' is the map induced by the abelianization map $F_n \to F_n/[F_n,F_n]=\Z^n$. In the examples that will be given, none of the homomorphisms restrict to the identity in the abelianization. Thus, these elements are not elements of $\IA_n$. {Equivalently, this says that the fundamental group of the base acts non-trivially on the homology of the fibre.}

\section{Examples}

\begin{ex}{Consider the groups }
$$G_1=\mathbb{Z}/2\mathbb{Z}:=\mathbb{Z}_2=\langle x_1|x_1^2=1\rangle \text{ and }G_2=\mathbb{Z}/3\mathbb{Z}:=\mathbb{Z}_3=\langle x_2|x_2^3=1\rangle.$$ Then using the Denham--Suciu fibration there is the following short exact sequence of groups
\[1 \longrightarrow F_2 \longrightarrow \mathbb{Z}_2 \ast \mathbb{Z}_3 \longrightarrow \mathbb{Z}_2 \times \mathbb{Z}_3 \longrightarrow 1,\]
where $F_2$ is the free group on the generators $\omega_1=[x_1,x_2]$ and $\omega_2=[x_1,x_2^2]$. 

To compute the map $\Theta:\Z_2 \ast \Z_3 \to \Aut(F_2)$, we first compute the automorphism  $\varphi_{x_1}\in \Aut(F_2)$ by looking at the image of the generators $\omega_1,\omega_2 \in F_2$ under $\varphi_{x_1}$ to find
 \[x_1\omega_1x_1^{-1}=[x_2,x_1]=([x_1,x_2])^{-1}=\omega_1^{-1}\]
 and
 \[x_1\omega_2x_1^{-1}=[x_2^2,x_1]=([x_1,x_2^2])^{-1}=\omega_2^{-1}.\]
 Looking at the induced map of $\varphi_{x_1}$ onto the abelianization $\mathbb{Z}\oplus \mathbb{Z}\cong F_2/[F_2,F_2]$, then
 \[\widetilde{\varphi}_{x_1}(\omega_1,\omega_2)=(-\omega_1,-\omega_2),\]
 which is given by the matrix
 \[ [\widetilde{\varphi}_{x_1}]=
 \begin{pmatrix}
 -1 & 0 \\
 0 & -1
 \end{pmatrix}
 \]
with respect to the basis $\{\omega_1,\omega_2\}$. {Then the other representations can be given by $[\widetilde{\varphi}_{x_1^i}]=[\widetilde{\varphi}_{x_1}]^i$.} Similarly, one can compute $\varphi_{x_2}\in \Aut(F_2)$ by finding
 \[x_2\omega_1x_2^{-1}=x_2[x_1,x_2]x_2^{-1}=[x_2,x_1][x_1,x_2^2]=\omega_1^{-1} \omega_2\]
and
 \[x_2\omega_2x_2^{-1}=[x_2,x_1]=([x_1,x_2])^{-1}=\omega_1^{-1}. \]
Looking at the induced map of $\varphi_{x_2}$ onto the abelianization $\mathbb{Z}\oplus \mathbb{Z}\cong F_2/[F_2,F_2]$, then
 \[\widetilde{\varphi}_{x_2}(\omega_1,\omega_2)=(-\omega_1+\omega_2,-\omega_1), \]
 which is given by the matrix
 \[ [\widetilde{\varphi}_{x_2}]=
 \begin{pmatrix}
 -1 & 1 \\
 -1 & 0
 \end{pmatrix}
 \]
with respect to the basis $\{\omega_1,\omega_2\}$. {Similarly, $[\widetilde{\varphi}_{x_2^i}]=[\widetilde{\varphi}_{x_2}]^i$}. Using properties of group actions, any automorphism $\varphi_{g}, g\in F_2$ can be found using $\varphi_{x_1}$ and $\varphi_{x_1}$. For example $\varphi_{x_1 x_2} = \varphi_{x_1} \circ \varphi_{x_2}$ and so on. Note that $\varphi_{x_1}$ and $\varphi_{x_2}$ are not elements of $\IA_2$ since the functions do not restrict to the identity in the abelianization. {Hence, the fundamental group of the base $\mathbb{Z}_2 \times \mathbb{Z}_3$ acts non-trivially on the homology of the fibre, as mentioned in the previous section}.
 
This calculation gives a homomorphism $\text{\rm ab}\circ \Theta:\mathbb{Z}_2 \ast \mathbb{Z}_3 \to GL_2(\mathbb{Z})$ by composing the homomorphisms 
\[\mathbb{Z}_2 \ast \mathbb{Z}_3 \overset{\Theta}{\longrightarrow} \Aut(F_2) \overset{\text{\rm ab}}{\longrightarrow} GL_2(\Z). \] 
The map $\Theta$ induces a homomorphism $\widetilde{\Theta}: \mathbb{Z}_2 \times \mathbb{Z}_3 \longrightarrow \Out(F_2)$. Moreover, the map $\text{\rm ab}\circ \Theta$ can be considered the same as the composition $p \circ \widetilde{\Theta}$, where $p$ is the projection to the abelianization of $\mathbb{Z}_2 \ast \mathbb{Z}_3 $, since $ [\widetilde{\varphi}_{x_1}]$ and $[\widetilde{\varphi}_{x_2}]$ commute.
\end{ex}

\begin{ex}
Let $\Sigma_3$ be the symmetric group on three letters, given by 
\[\Sigma_3=\{1,(12),(13),(23),(123),(132)\}. \]
Let $C_2=\Z_2=\{1,x\}$ be the cyclic group with two elements. There is a short exact sequence of groups
\[1 \longrightarrow F_5 \longrightarrow \Z_2 \ast \Sigma_3 \longrightarrow \Z_2 \times \Sigma_3 \longrightarrow 1, \]
{where $F_5$ is the free group on letters $W=\{[x,g]|x,g\neq 1, x\in \Z_2, g \in \Sigma_3\}$. To calculate the representation $\Z_2 \times \Sigma_3 \to \Out(F_5)$, start with evaluating $\varphi_x$ for $x\in \Z_2$. Hence,
$\varphi_x([x,g])=[g,x]=[x,g]^{-1}$ for all $g\in \Sigma_3$. After restricting to the abelianization $\widetilde{\varphi}_x([x,g])=-[x,g]$. Hence, the matrix representation of $\widetilde{\varphi}_x$ is given by $  [\widetilde{\varphi}_{x}] =-I_5$.}

Similarly, $\varphi_{(12)}([x,(12)])=[(12),x]$ and $\varphi_{(12)}([x,g])=[(12),x][x,(12)\cdot g]$ if $g \neq (12)$. In the abelianization, we get $\widetilde{\varphi}_{(12)}([x,(12)])=-[x,(12)]$ and $\widetilde{\varphi}_{(12)}([x,g])=-[x,(12)]+[x,(12)g]$. Order the basis as follows
\[W=\{[x,(12)],[x,(13)],[x,(23)],[x,(123)],[x,(132)]\}. \]
Then the matrix representation for $\widetilde{\varphi}_{(12)}$ is 
\[
[\widetilde{\varphi}_{(12)}]=
\begin{pmatrix}
-1 & -1 & -1 & -1 & -1 \\
0 &  0 & 0 & 0  & 1 \\
0  & 0 & 0 & 1  & 0 \\
0 &  0 & 1 & 0 & 0 \\
0 &  1 & 0 & 0 & 0 
\end{pmatrix}.
\]
One can find the other automorphisms similarly, since  $\varphi_g([x,h])=[g,x][x,gh]$. Hence,
\[
[\widetilde{\varphi}_{(13)}]=
\begin{pmatrix}
0 &  0 & 0 & 1  & 0 \\
-1 & -1 & -1 & -1 & -1 \\
0  & 0 & 0 & 0  & 1 \\
1 &  0 & 0 & 0 & 0 \\
0 &  0 & 1 & 0 & 0 
\end{pmatrix}
,
[\widetilde{\varphi}_{(23)}]=
\begin{pmatrix}
0 &  0 & 0 & 0  & 1 \\
0  & 0 & 0 & 1  & 0 \\
-1 & -1 & -1 & -1 & -1 \\
0 &  1 & 0 & 0 & 0 \\
1 &  0 & 0 & 0 & 0 
\end{pmatrix}
\]
and
\[
[\widetilde{\varphi}_{(123)}]=
\begin{pmatrix}
0 &  0 & 1 & 0  & 0 \\
1  & 0 & 0 & 0  & 0 \\
0 &  1 & 0 & 0 & 0 \\
-1 & -1 & -1 & -1 & -1 \\
0 &  0 & 0 & 1 & 0 
\end{pmatrix}
, [\widetilde{\varphi}_{(132)}]=
\begin{pmatrix}
0 &  1 & 0 & 0  & 0 \\
0  & 0 & 1 & 0  & 0 \\
1 &  0 & 0 & 0 & 0 \\
0 &  0 & 0 & 0 & 1 \\
-1 & -1 & -1 & -1 & -1 
\end{pmatrix}.
\]
Note that these matrices do not commute in general. For example $[\widetilde{\varphi}_{(13)}]\cdot [\widetilde{\varphi}_{(132)}] \neq [\widetilde{\varphi}_{(132)}]\cdot [\widetilde{\varphi}_{(13)}]$. However, $[\widetilde{\varphi}_{x}]$  commutes with the other matrices. Hence, the map
\[\Z_2 \ast \Sigma_3 \overset{\Theta}{\longrightarrow} \Aut(F_5) \overset{\text{\rm ab }}{\longrightarrow} GL_5(\Z)\]
is the same as the composition 
\[\Z_2 \ast \Sigma_3 \overset{p}{\longrightarrow} \Z_2 \times \Sigma_3 \overset{\widetilde{\Theta}}{\longrightarrow} \Out(F_5) \overset{\text{\rm ab }}{\longrightarrow} GL_5(\Z). \]
Therefore, there is a homomorphism $\Z_2 \times \Sigma_3 \to GL_5(\Z)$. Also note that 
$$det([\widetilde{\varphi}_{(12)}])=det([\widetilde{\varphi}_{(13)}])=det([\widetilde{\varphi}_{(23)}])=-1,$$
and consequently
$$
det([\widetilde{\varphi}_{(123)}])=det([\widetilde{\varphi}_{(132)}])=1.
$$
\end{ex}

\

\subsection{{Two finite cyclic groups}}\label{section:two cyclic groups}

\
 
\ 
 
{In this section we prove Theorem \ref{thm: 2 cyclic gps rep}.}

\begin{proof}[{Proof of Theorem \ref{thm: 2 cyclic gps rep}}] 
Consider the general case of two cyclic groups
$$G_1\cong \mathbb{Z}/{r}\mathbb{Z}\cong \langle x_1|x_1^ {r}=1\rangle\text{ and }G_2\cong \mathbb{Z}/m\mathbb{Z}\cong \langle x_2|x_2^m=1 \rangle.$$ 
There is a short exact sequence of groups
 \[ 1 \longrightarrow F_k \longrightarrow \mathbb{Z}_{r} \ast \mathbb{Z}_m \longrightarrow \mathbb{Z}_{r} \times \mathbb{Z}_m \longrightarrow 1, \]
coming from the Denham--Suciu fibration, where $F_k$ is the free group on $k=({r}-1)(m-1)$ letters given by the elements of
 \[\mathcal{W}_2 = \{ \omega_{ij}=[x_1^i,x_2^j]|1\leq i \leq {r}-1, 1\leq j \leq m-1 \}. \]
 
To compute the map $\Theta:\Z_{r} \ast \Z_m \to \Aut(F_k)$, we first compute the automorphism  $\varphi_{x_1}\in \Aut(F_k)$ by looking at the image of the generators $\omega_{ij} \in F_k$ under $\varphi_{x_1}$ to find 
 \[x_1\omega_{ij}x_1^{-1}=x_1[x_1^i,x_2^j]x_1^{-1}= [x_1^{i+1},x_2^j][x_2^j,x_1] = \omega_{i+1,j} {\omega_{1,j}}^{-1}. \]
 
Looking at the induced map of $\varphi_{x_1}$ onto the abelianization 
 \[\bigoplus_{({r}-1)(m-1)}\mathbb{Z}\cong F_{({r}-1)(m-1)}/[F_{({r}-1)(m-1)},F_{({r}-1)(m-1)}]\] 
then
 \[\widetilde{\varphi}_{x_1}(\omega_{11},\dots,\omega_{({r}-1)(m-1)})=(\omega_{2,1}-\omega_{1,1},\omega_{2,2}-\omega_{1,2}, \dots,-\omega_{({r}-1),(m-1)})\]
 which is given by the matrix 
 \begin{equation}\label{monodromy equation 1}
 [\widetilde{\varphi}_{x_1}]=
 \begin{pmatrix}
 -I_{m-1} & I_{m-1} & 0 & 0 & \cdots & 0 \\
 0 & -I_{m-1} & I_{m-1} & 0 & \cdots & 0 \\
 0& 0 & -I_{m-1} & I_{m-1} & \cdots & 0 \\
 \vdots & \vdots  & \vdots & \ddots &\ddots & \vdots \\
 0 & \cdots & 0 & 0 & -I_{m-1} & I_{m-1} \\
 0 & \cdots & 0 & 0 & 0\cdots & -I_{m-1}
 \end{pmatrix}
 \end{equation}
with respect to the basis $\mathcal{W}_2$, where $I_{m-1}$ is the $(m-1)\times(m-1)$ identity matrix. Hence, clearly ${\varphi}_{x_1}$ is not an element of $\IA_{k}$.
 
 For $\varphi_{x_2}\in \Aut(F_{k})$: 
 \[x_2\omega_{ij}x_2^{-1}=x_2[x_1^i,x_2^j]x_2^{-1}= [x_2,x_1^i][x_1^{i},x_2^{j+1}] = {\omega_{i,1}}^{-1} \omega_{i,j+1}. \]
Similarly, looking at the induced map of $\varphi_{x_2}$ onto the abelianization of $F_k$ we get
 \[\widetilde{\varphi}_{x_2}(\omega_{11},\dots,\omega_{({r}-1)(m-1)})=(-\omega_{1,1}+\omega_{1,2}-\omega_{1,1}+\omega_{1,3}, \dots,-\omega_{({r}-1),(m-1)}), \]
 which is given by the matrix
 \begin{equation}\label{monodromy equation 2}
 [\widetilde{\varphi}_{x_2}]=
 \begin{pmatrix}
 A_1 & 0 & \cdots & 0 \\
 0 & A_2 & \cdots & 0 \\
 \vdots & \vdots & \ddots & \vdots \\
 0 & 0 & \cdots & A_{{r}-1}
 \end{pmatrix}
 \end{equation}
 with respect to the basis $\mathcal{W}_2$, where 
 \[ A_i= 
 \begin{pmatrix}
 -1 & 1 & 0 & 0 & \cdots & 0 \\
 -1 & 0 & 1 & 0 &\cdots & 0 \\
 -1 & 0 & 0 & 1 &\cdots & 0 \\
 -1 & 0 & 0 & 0 &\ddots & 0  \\
 \vdots & \vdots & \vdots  & \vdots & \cdots & 1\\
 -1 & 0 & 0 & 0 & \cdots & 0 
 \end{pmatrix}_{(m-1)\times (m-1)}
 \]
for all $i$. Hence, ${\varphi}_{x_2}$ is not an element of $\IA_{k}$. 
 
In general, $\Theta$ maps an element $x_1^i x_2^j$ to $\varphi_{x_i \cdot x_j}\in \Aut(F_{k})$, which when restricted to the abelianization $\bigoplus_k \mathbb{Z}$, can be identified with the matrix $[\widetilde{\varphi}_{x_1}]^i[\widetilde{\varphi}_{x_2}]^j$. This matrix is the identity if and only if $i={r}$ and $j=m$. Hence, there is a homomorphism
\[\mathbb{Z}_{r} \ast \mathbb{Z}_m  \xrightarrow{\text{\rm ab} \circ \Theta} GL_{k}(\mathbb{Z}). \]
$\Theta$ induces a homomorphism $\widetilde{\Theta}:\Z_{r} \times \Z_m\longrightarrow \Out(F_k)$. Hence, there is a homomorphism
$\mathbb{Z}_{r} \times \mathbb{Z}_m \xrightarrow{\text{\rm ab} \circ \widetilde{\Theta}}  GL_{k}(\mathbb{Z})$.

If $m$ is even and ${r}$ is odd or vice versa, then 
\[ det [\widetilde{\varphi}_{x_1}] = (-1)^{({r}-1)(m-1)}=1, \]
\[det [\widetilde{\varphi}_{x_2}] = det(A_1)\cdots det(A_{{r}-1})=(det(A_1))^{{r}-1}. \]
Since $det(A_i)=1$ if $m$ is odd and -1 if $m$ is even, and if ${r}$ is odd we get $(-1)^{{r}-1}=1$, then $det[\widetilde{\varphi}_{x_2}]=1$. Hence, there is a homomorphism
\[\mathbb{Z}_{r} \ast \mathbb{Z}_m \longrightarrow SL_{k}(\mathbb{Z})\] 
which induces a homomorphism
\begin{center}
\begin{tikzcd}
\mathbb{Z}_{r} \ast \mathbb{Z}_m \arrow{dr}{\text{\rm ab}} \arrow[dotted]{rrr}{{\text{\rm ab}}\circ \widetilde{\Theta} \circ {\text{\rm ab}}}& & & SL_{k}(\mathbb{Z})\subset GL_k(\Z)\\
 & \mathbb{Z}_{r} \times \mathbb{Z}_m \arrow{r}{\widetilde{\Theta}} & \Out(F_k).  \arrow{ur}{\text{\rm ab}} & 
\end{tikzcd}
\end{center}
That is, there is a representation of $\mathbb{Z}_{r} \times \mathbb{Z}_m \to SL_{k}(\mathbb{Z})$. Similarly as before, the map $\text{\rm ab}\circ \Theta$ can be considered the same as the composition $p \circ \widetilde{\Theta}$, where $p$ is the projection to the abelianization of $\mathbb{Z}_{r} \ast \mathbb{Z}_m $, since $ [\widetilde{\varphi}_{x_1}]$ and $[\widetilde{\varphi}_{x_2}]$ commute. To show that $ [\widetilde{\varphi}_{x_1}]$ and $[\widetilde{\varphi}_{x_2}]$ commute it suffices to show that they commute for ${r}=m=3$
\[  [\widetilde{\varphi}_{x_1}]\cdot[\widetilde{\varphi}_{x_2}] = [\widetilde{\varphi}_{x_2}]\cdot[\widetilde{\varphi}_{x_1}] =
\begin{pmatrix}
 1 & -1 & -1 & 1  \\
 1 & 0 & -1 & 0  \\
 0  & 0 & -1 & 1  \\
 0 & 0 & -1 & 0  
 \end{pmatrix}.
 \]
\end{proof}

\subsection{{Two arbitrary finite groups}}\label{section:two finite groups}

\

\

{
We begin by proving Proposition \ref{proposition: faithful rep of GxH to Out(Fn) INTRO}.}
\begin{proof}[{Proof of Proposition \ref{proposition: faithful rep of GxH to Out(Fn) INTRO}}]
Let $G$ and $H$ be finite discrete groups, not necessarily cyclic or abelian, with cardinality $m$ and $n$ respectively. That is, assume 
$$G=\{1,g_1,\dots,g_{m-1}\}\text{ and }H=\{1,h_1,\dots,h_{n-1}\}.$$ 
There is a short exact sequence of groups
\[1 \longrightarrow F_{(m-1)(n-1)} \longrightarrow G \ast H \longrightarrow G \times H \longrightarrow 1\]
coming from the Denham--Suciu fibration, {where the rank of the free group is determined by the formula in equation \ref{eqn: formula for N_n}}. To calculate the map 
$$G \ast H \to \Aut(F_{(m-1)(n-1)}),$$ start with $\varphi_f$, where $f\in G$ or $f \in H$. Choose a basis for $F_{(m-1)(n-1)}$ to be 
\[W=\{[g_i,h_j]| 1 \leq i \leq m-1, 1\leq m \leq n-1\}.\] 
Then, 
\begin{equation}\label{monodromy for G and H}
\begin{split}
&\varphi_{g_k}([g_i,h_j])={g_k}[g_i,h_j]{g_k}^{-1}=[g_k g_i,h_j][h_j,g_k]  \\ 
&\varphi_{h_k}([g_i,h_j])={h_k}[g_i,h_j]{h_k}^{-1}=[h_k,g_i][g_i,h_k h_j].
\end{split}
\end{equation}
Note that the images $\varphi_{g_k}([g_i,h_j])$ and $\varphi_{h_k}([g_i,h_j])$ are trivial if and only if $g_k=1$ and $h_k=1$, respectively. Therefore, the representation is faithful.
\end{proof}

To find the matrix representation of these, it is necessary to know the group structure of $G$ and $H$. Hence, we get a composition of homomorphisms 
\[G\ast H \to G\times H \to \Out(F_{(m-1)(n-1)})\]
which is the same as the composition
\[G\ast H \to \Aut(F_{(m-1)(n-1)}) \to \Out(F_{(m-1)(n-1)}).\]

\begin{rmk}
In sections \ref{section:two cyclic groups} and \ref{section:two finite groups} as well as in the examples, data is being collected, with the goal of axiomatizing properties of the monodromy.
\end{rmk}

\begin{prop}\label{proposition: faithful rep of GxH to Out(Fn)}
Let $G$ and $H$ be two finite discrete groups with with cardinality $m$ and $n$, respectively. Then there is a faithful a representation 
$$
G\times H \to \Out(F_{k}),
$$
given by equation \ref{monodromy for G and H}, where $k=(m-1)(n-1)$.
\end{prop}
\

\subsection{{A collection of finite discrete groups}}\label{subsection:monodromy.finite.collection.gps}

\

\

Recall that for a group $G$, there is a sequence of subgroups called the \textit{descending central series} of $G$ given by
\[G=\Gamma^1(G) \unrhd \Gamma^2(G) \unrhd \cdots \unrhd \Gamma^n(G) \unrhd \cdots \]
such that the second stage is $\Gamma^2(G)=[G,G]$ and the $(n+1)$--st stage is given inductively by $\Gamma^{n+1}(G)=[\Gamma^{n}(G),G]$. The Lie algebra of $G$ associated to the descending central series is given by 
\[\gr_{\ast}(G)= \bigoplus_{i\geq 1} \Gamma^{i}(G)/\Gamma^{i+1}(G)\]
with $\gr_{p}(G)=\Gamma^{p}(G)/\Gamma^{p+1}(G)$. 

\begin{lemma}
Let $\{G_i\}_{i=1}^n$ be a collection of finite discrete groups and $K_0$ be the $0$--simplicial complex on $n$ vertices. Let $\rho :\prod_{i=1}^n G_i \to  \Out(F_N)$ be the monodromy representation where $F_N$ is isomorphic to the kernel of the projection $p: G_1 \ast \cdots \ast G_n \to \prod_{i=1}^n G_i$. Then the following hold:
\begin{enumerate}
\item There is a choice of a generating set for $F_N$ that consists of elements of the form
\[f=[g_{i_1},[g_{i_2},[\dots,[g_{i_{k-1}},g_{i_{k}}]\dots]]] \in \Gamma^k(G_1 \ast \cdots \ast G_n)\]
such that $g_{i_j} \in G_{i_j}$, for all $i_j$.
\item For any $g \in G_1 \ast \cdots \ast G_n$, the map $\rho(g) \in \Aut(F_N)$ satisfies $\rho(g)(f)=\Delta \cdot f$, {where} $\Delta \in \Gamma^{k+1}(G_1 \ast \cdots \ast G_n)$. That is, $\Delta$ is trivial in $ \gr_{p}(G_1 \ast \cdots \ast G_n)$ for $p \leq k$.
\end{enumerate}
\end{lemma}
\begin{proof} Part 1: From the homotopy type of $Z_{K_0}(\underline{EG},\underline{G})\subset [0,1]^n$ it is clear that all types of paths can be described using commutators of length at most $n$. It remains to prove that it is sufficient to consider only $g_{i_j}\in G_{i,j}$ and not other elements in $G_1 \ast \cdots \ast G_n$ to construct these commutators. 

Start with $[g_i g_j, g_k] \in \Gamma^3(G_1 \ast \cdots \ast G_n)$. Then
\[[g_i g_j, g_k]=[g_i ,[g_j, g_k]]\cdot [g_j, g_k\cdot ][g_i, g_k].\]
Thus for any product, say $g_i=h_1 \cdots h_t$, it follows that
\[[g_i g_j, g_k]=[(h_1 \cdots h_t) g_j, g_k]=[h_1 \cdots h_t ,[g_j, g_k]]\cdot [g_j, g_k]\cdot [h_1 \cdots h_t, g_k]. \]
Then this product can be reduced to a product of commutators of the form stated in part 1, in finitely many steps by applying the step $t$ more times.

Part 2: If $f=[g_{i_1},[g_{i_2},[\dots,[g_{i_{k-1}},g_{i_{k}}]\dots]]] \in \Gamma^k(G_1 \ast \cdots \ast G_n)$ is an element in $F_N$, then
\begin{align*}
\rho(g)(f)& =g \cdot [g_{i_1},[g_{i_2},[\dots,[g_{i_{k-1}},g_{i_{k}}]\dots]]] \cdot g^{-1}\\
& = [g, [g_{i_1},[g_{i_2},[\dots,[g_{i_{k-1}},g_{i_{k}}]\dots]]]] \cdot [g_{i_1},[g_{i_2},[\dots,[g_{i_{k-1}},g_{i_{k}}]\dots]]]\\
& = \Delta \cdot f,
\end{align*}
where $\Delta = [g, [g_{i_1},[g_{i_2},[\dots,[g_{i_{k-1}},g_{i_{k}}]\dots]]]] = [g,f] \in \Gamma^{k+1}(G_1 \ast \cdots \ast G_n)$.
\end{proof}


{
Finally, in the following remark we discuss the implications that these representations might have for the monodromy for any flag complex $K$.



\begin{center}
***
\end{center}

\begin{rmk}Consider the Denham and Suciu fibration for flag complexes $K$ and finite discrete groups $G_1,\dots,G_n$, and consider the corresponding monodromy representation
\[
{\rho_{K}}:G_1\times \cdots \times G_n \longrightarrow  \Out(\pi_1(Z_K(\underline{EG},\underline{G}))).
\]
We are interested in a possible relation between $\rho_K$ and $$
{\rho_{K_0}}:G_1\times \cdots \times G_n \longrightarrow  \Out(F_N)
,$$
where $F_N$ is the kernel of the projection $G_1 \ast \cdots \ast G_n \twoheadrightarrow G_1 \times \cdots \times G_n.$ 
We are lead to believe that solving $\rho_{K_0}$ will help solve the other representations $\rho_{K}$ because of the geometric description of monodromy. The action of the fundamental group of the base shifts loops of the fibre in a \textit{certain direction}. On the other hand adding higher dimensional faces to $K_0$ will kill loops in the fibre in a way that can be described precisely (e.g. adding an edge kills loops \textit{parallel} to each other etc.). Then the monodromy for the new $K$ can be described, at least geometrically, using the monodromy for $K_0$. As an illustration, figures \ref{fig:K_0} and \ref{fig:K filled} give the fibre of the Denham-Suciu fibration for the choice of $G_1=G_2=G_3=\Z/2\Z$ with $K_0$ and $K$, respectively, where $K_0$ has only three vertices and $K$ has an extra edge. Here $\rho_{K_0}$ is supposed to determine $\rho_{K}$ since the loops in the shaded faces are killed.

\begin{figure}[ht!]
\centering
\begin{minipage}{0.5\linewidth}
\centering
{
\begin{tikzpicture}[scale=.94]
\path[draw](0,0)--(3,0);
\path[draw](3,0)--(5,1);
\path[draw](5,1)--(2,1);
\path[draw](2,1)--(2,3);
\path[draw](2,3)--(5,3);
\path[draw](5,3)--(3,2);
\draw[solid](3,2)--(0,2);
\draw[solid](0,0)--(0,2);
\draw[solid](0,0)--(2,1);
\draw[solid](3,0)--(3,2);
\draw[solid](5,1)--(5,3);
\draw[solid](0,2)--(2,3);
\coordinate [label=below left:$0$] (1) at (0,0);
\coordinate [label=below:$1$] (5) at (3,0);
\coordinate [label=above left:$1$] (5) at (2,1);
\coordinate [label=left:$1$] (5) at (0,2);
\end{tikzpicture}
}
\caption{$G_i=\Z/2\Z$, $K_0$}\label{fig:K_0}
\end{minipage}%
\begin{minipage}{0.5\linewidth}
\centering
{
\begin{tikzpicture}[scale=.94]
\path[draw](0,0)--(3,0);
\path[draw](3,0)--(5,1);
\path[draw](5,1)--(2,1);
\path[draw](2,1)--(2,3);
\path[draw](2,3)--(5,3);
\path[draw](5,3)--(3,2);
\draw[solid](3,2)--(0,2);
\draw[solid](0,0)--(0,2);
\draw[solid](0,0)--(2,1);
\draw[solid](3,0)--(3,2);
\draw[solid](5,1)--(5,3);
\draw[solid](0,2)--(2,3);
\coordinate [label=below left:$0$] (1) at (0,0);
\coordinate [label=below:$1$] (5) at (3,0);
\coordinate [label=above left:$1$] (5) at (2,1);
\coordinate [label=left:$1$] (5) at (0,2);
\filldraw [draw=black,fill=gray,opacity=0.5]
(0,0)--(3,0)--(5,1)--(2,1)--cycle ;
\filldraw [draw=black,fill=gray,opacity=0.5]
(0,2)--(3,2)--(5,3)--(2,3)--cycle ;
\end{tikzpicture}
}
\caption{$G_i=\Z/2\Z$, $K$}\label{fig:K filled}
\end{minipage}
\end{figure}
A similar situation would occur in higher dimensions, where if the faces $\sigma_1,\dots,\sigma_s$ are added to $K_0$ to obtain $K$, the monodromy for $K$  would be extracted from the monodromy for $K_0$, by keeping track of the order in which the faces are added and which loops are killed. This is believed to work \textit{a priori} since monodromy shifts non-trivial loops in the direction where the loops are not killed.

One way to attack this problem algebraically is as follows: there is a commutative diagram of fibrations
\begin{equation}\label{diagram: monodromy K0 vs K}
\begin{tikzcd}
H(p) \arrow{r}\arrow{d}  & H(p) \arrow{r}\arrow{d}  & \ast \arrow{d}  \\
Z_{K_0}(\underline{EG},\underline{G}) \arrow{r}\arrow{d}[swap]{p}  & Z_{K_0}(\underline{BG}) \arrow{r}\arrow{d}  & \prod_{i=1}^n{BG_i} \arrow{d}\\
Z_K(\underline{EG},\underline{G}) \arrow{r} & Z_K(\underline{BG}) \arrow{r} & \prod_{i=1}^n {BG_i},
\end{tikzcd}
\end{equation}
where $H(p)$ is the homotopy fibre of the map $p$. The fibre $H(p)$ is connected, so it follows from the long exact sequence in homotopy that the map $p$ induces a surjection 
$$p_{\#}: \pi_1(Z_{K_0}(\underline{EG},\underline{G})) \to \pi_1(Z_{K}(\underline{EG},\underline{G}))$$
on the level of fundamental groups. Thus, the kernel of the projection map is a free group, say $F_q$.
From \cite[Theorem 1.1]{stafa.fund.gp} it follows that both the fibre $H(p)$ and $Z_{K}(\underline{EG},\underline{G})$ are Eilenberg--Mac Lane spaces. Assume $Z_{K}(\underline{EG},\underline{G})$ has fundamental group $\pi$.

Let $F_N$ be the kernel of the projection $G_1 \ast \cdots \ast G_n  \to \prod_{i=1}^n G_i$. Consider the commutative diagram of fibrations in (\ref{diagram: monodromy K0 vs K}). If $K$ is a flag complex, then there is a commutative diagram as follows

\begin{equation}
\begin{tikzcd}[row sep=scriptsize, column sep=0]
F_q \arrow{dd}\arrow{rr}{\cong} &  & F_q \arrow{rr}\arrow{dd} &  & 1 \arrow{dd}  & \\
 & & & & & \\
F_N \arrow{rr} \arrow{dd}[swap]{p_{\#}} \arrow{dr} &  & G_1 \ast \cdots \ast G_n \arrow{rr} \arrow{dd} \arrow{dr} &  & \prod G_i \arrow{dd} \arrow{dr}{\rho_{K_0}} & \\
	&  \Inn(F_N)  \arrow[crossing over]{dd} \arrow[crossing over]{rr}&  & \Aut(F_N) \arrow[crossing over]{rr}  &   & \Out(F_N) \arrow[dotted]{dd} \\
\pi \arrow{rr}  \arrow{dr} &   & \prod_{SK_1} G_i \arrow{rr}  \arrow{dr} &   & \prod G_i  \arrow{dr}{\rho_{K}} & \\
	& \Inn(\pi) \arrow[crossing over, leftarrow]{uu} \arrow{rr}  &   & \Aut(\pi) \arrow[crossing over, leftarrow, dotted]{uu} \arrow{rr}   & & \Out(\pi) .
\end{tikzcd}
\end{equation}
where the dotted homomorphisms are yet to be determined if they exist. The goal is to show that if there is a homomorphism $r: \Out(F_N) \to \Out(\pi)$ induced by $p_{\#}$, then there is a homomorphism $\rho_K :G_1 \times \cdots \times G_n \to \Out(\pi)$  such that the following diagram commutes
\begin{center}
\begin{tikzcd}
{\color{white}{1}} & 
\Out(F_N)  \arrow[dotted]{dd}{r} \\
G_1 \times \cdots \times G_n  \arrow{ur}{\rho_{K_0}}\arrow{dr}[swap]{\rho_{K}} & \\
& \Out(\pi) .
\end{tikzcd}
\end{center}
That means, $\rho_K = r \circ {\rho_{K_0}} $. Hence, we want to find such a map $r$. 
$\blacksquare$
\end{rmk}
}

\section*{Acknowledgments}
The author would like to thank Fred Cohen for his suggestions. 

The author is supported by DARPA grant number N66001-11-1-4132.

\end{document}